\documentclass{llncs}
\usepackage{amsmath}
\usepackage{amssymb}
\usepackage{graphicx}
\usepackage{epsfig}

\def\real{\hbox{\rm\vrule\kern-1pt R}}
\def\nat{\hbox{\rm\vrule\kern-1pt N}}

\def\2kpt{(KPT2)}

\def\c2kpt{(C(KPT2))}

\usepackage{epsfig}

\begin{document}


\title{Radio resource allocation in OFDMA multi-cell networks}
\author{
         {Paolo Detti}\inst{1}
    \and {Marco Moretti}\inst{2}
    \and{Andrea Abrardo}\inst{1}
    }

\institute{
           {Universit\`a degli Studi di Siena,
            Dipartimento di Ingegneria dell'Informazione,
           Siena, Italia;
           e-mail: {\tt \{abrardo,detti\}@dii.unisi.it}
           }
           \and
           {Dipartimento di Ingegneria dell'Informazione, Universit\`a
            di Pisa, Italia;
            e-mail {\tt moretti@iet.unipi.it}
           }
           }

\maketitle


\begin{abstract}
In this paper, the problem of allocating users  to radio resources
(i.e., subcarriers) in the downlink of an OFDMA cellular network
is addressed. We consider a multi-cellular environment with a
realistic interference model and a \emph{margin adaptive}
approach, i.e., we aim at minimizing total transmission power
while maintaining a certain given rate for each user.
The computational complexity issues of the  resulting model is
discussed and proving that the problem is NP-hard in the strong
sense. Heuristic approaches, based on network flow models, that
finds optima under suitable conditions, or ``reasonably good''
solutions in the general case are presented. Computational
experiences show that, in a comparison with a commercial
state-of-the-art optimization solver, the proposed algorithms are
effective in terms of solution quality and CPU times.


\end{abstract}

\section {Introduction}
\label{sec:intro}
%

Resource allocation (RA) is one of the most efficient techniques
to increase the performance of multicarrier systems. In an
orthogonal frequency-division multiple access (OFDMA) scheme, each
user of a communication system is allocated a different subset of
orthogonal subcarriers of a radio frame. If the transmitter
possesses full knowledge of channel state information, the
subcarriers can be assigned to the various users following a
certain optimality criterion to increase the overall spectral
efficiency. In fact, propagation channels are independent for each
user and thus the subcarriers that are in a deep fade for one user
may be good ones for another and the goal is to assign only high
quality  channels to each different user to exploit the so-called
\emph{multiuser diversity}. One of the major drawback of efficient
RA schemes is that their complexity is in general high and tends
to grow larger with the number of users and subcarriers. Many
resource allocation algorithms have been designed for taking
advantage of both the frequency selective nature of the channel
and the multi-user diversity \cite{Wong}-\cite{Jang} but all of
them exhibit a trade-off between complexity and performance:  low
complexity \cite{Kivanc,Rhee} algorithms tend to be outperformed
by those requiring high computational loads \cite{Wong,Kim}.

In this paper, we consider  the downlink of a multi-cellular OFDMA
system, and we deal with the assignment both of radio resources
(i.e., subcarriers) of a radio frame and of transmission formats
to users minimizing the overall transmission power, while
providing a given transmission rate for each user. The allocation
is performed on a radio frame basis, and we assume  that the
propagation channel is quasi static, i.e., it does not vary within
a radio frame. As a consequence, even if the problem is
intrinsically dynamic, for the allocation decisions within a radio
frame, data may be regarded as static.

The paper is organized as follows. Section \ref{sec: sys_model}
describes the system model and defines the RA problem. Section
\ref{sec:comp} addresses the computational complexity of the
problem. In particular, it is proved that the addressed problem is
NP-hard in the strong sense. Heuristic approaches, based on
network flow models, that finds optima under suitable conditions,
or ``reasonably good'' solutions in the general case are presented
in Section  \ref{sec:algo}. In Section \ref{sec:exp}, experimental
results are presented and discussed. Finally, Section
\ref{sec:conc} provides conclusive remarks.

\section{System model and problem definition}\label{sec: sys_model}
The problem we address, that we call RRAP, is a constrained
minimization problem in which subcarriers and transmission formats
must be assigned to the users, in such a way that a given bit-rate
is provided to each user and that the total transmission power is
minimized. In particular, we consider a multi-cell scenario with
users belonging to different cells and a frequency
bandwidth divided into orthogonal subcarriers.\\

Let $N=\{1,\ldots,n\}$ and $C=\{1,\ldots,c\}$ be the sets of the
users and cells, respectively, and let $N_k$ be the set of users
in cell $k$. (Hence, $N=\{N_1\cup\ldots\cup N_c\}$.) Each user
belongs to exactly one cell, and, given a user $i\in N$, we denote
by $b(i)$ the cell of user $i$. Let $M=\{1,\ldots,m\}$ be the set
of the available subcarriers and $Q=\{q_1,\ldots,q_p\}$ be the set
of possible transmission formats.

Let $R_i$ be the required transmission rate of user $i$, $i\in N$.
A given transmission format $q$ corresponds to the usage of a
certain error correction code and symbol modulation that leads to
a spectral efficiency $\eta_q$:  a user employing format $q$ on a
certain subcarrier transmits with rate $R=B\eta_q$, $B$ being the
bandwidth of each subcarrier. The target Signal-to-interference
ratio (SIR) to achieve the spectral efficiency
$\eta_q=\log_2\left(1+SIR(q)\right)$ is $SIR(q)=2^{\eta_q} - 1$.

To simplify the resource allocation algorithm, we impose that all
user requests are expressed as an integer multiple of a certain
fixed rate $R_0 = B\eta_0$, i.e., the rate requested by user $i$
is $R_i = r_i R_0$ with $r_i$ an integer number. In the same way,
we assume that if user $i$ transmits on subcarrier $j$
transmission format $q\in Q$, she gets a transmission rate of
$B\eta_q=Bq\eta_0$. Hence, bigger is the transmission format
bigger is the transmission rate. Moreover, given a certain
$\eta_0$, the target SIR to achieve the spectral efficiency
$\eta_q$ is $SIR(q)=2^{\eta_q} - 1=2^{q\eta_0} - 1$.

In general, users belonging to different cells can share the same
subcarrier, while interference phenomena do not allow two users in
the same cell to transmit on the same subcarrier. However, the
power required for the transmission on a given subcarrier
increases both with the set of users transmitting on that
subcarrier and with the used transmission formats. More precisely,
let $S(j)$ be the set of users (belonging to different cells)
which are assigned to (i.e., that are transmitting on) the same
subcarrier $j$, and let $q_i\in Q$ be the transmission format of
user $i \in S(j)$ on subcarrier $j$. Let $p_{ijq}$ be the
transmission power needed to user $i$ for transmitting with format
$q$ on subcarrier $j$. The transmission powers $p_{ijq}$ have to
satisfy the following system.

\begin{eqnarray}\label{sist:sir}
p_{ijq_i}=SIR(q_i){  {\sum\limits_{h \in S(j), h\not= i}
G_{i}^{b(h)}(j)
p_{hjq_h}}+BN_0  \over G_{i}(j) }& \; \; \forall i \in S(j) \label{sist1r}\\
p_{ijq_i}\geq 0 &\; \; \forall i \in S(j) \label{sist12r}
\end{eqnarray}

\noindent where $G_{i}(j)$ is the channel gain of user $i$ on
subcarrier $j$, $G_{i}^k(j)$ is the channel gain between user $i$
and the base station of cell $k\not= b(i)$ on subcarrier $j$.
 \noindent Values $G_{i}^k(j)$ are a measure of the
interference between user $i$ and users of other cells
transmitting on the same subcarrier $j$. In Equation
(\ref{sist1r}),  we refer to the term $\sum\limits_{h \in S(j),
h\not= i} G_{i}^{b(h)}(j) p_{hjq_h}$ as to {\em interference
term}.

Note that, power $p_{ijq_i}$ increases as the interference term
increases, moreover, the interference term depends on the set of
users, other than $i$, which are assigned to the same subcarrier,
and by their transmission formats. Note also that, system
(\ref{sist1r})--(\ref{sist12r}) may not have a feasible solution.
On the other hand, if only user $i$ is assigned subcarrier $j$
(i.e., if the interference term is $0$), by (\ref{sist1r}),
$p_{ijq_i}={ SIR(q_i) BN_0 \over G_{i}(j)}$.

A {\em feasible radio resource allocation} for RRAP consists in
assigning subcarriers to users and, for each subcarrier-user pair,
in choosing a transmission format, in such a way that $(a)$ for
each user $i$, a bit-rate $R_i$ is achieved, $(b)$ the users in
the same cell are not assigned to the same radio resource (i.e.,
subcarrier), $(c)$ given the set $S(j)$ of users assigned to radio
resource $j$, System (\ref{sist1r})--(\ref{sist12r}) has a
feasible solution, for any subcarrier $j\in M$.

RRAP consists in finding a feasible radio resource allocation
minimizing the overall transmission power.

\subsection{A MILP formulation for RRAP}\label{sec:ILPform}
In this section, a Mixed Integer Linear Programming (MILP)
formulation of RRAP is presented. Let $x_{ijq}$ be a binary
variable equal to $1$ if user $i$ is assigned to subcarrier $j$
with format $q$ (and $0$ otherwise), and let $p_{ijq}$ be a
positive real variable denoting the transmission power allocated
for user $i$ on subcarrier $j$ with format $q$. RRAP can be
formulated as follows.

{\small
\begin{eqnarray}
 \min  \sum_{i\in N,j\in M, q \in Q}p_{ijq}&& \label{fob}\\
p_{ijq}\leq L \, x_{ijq}\; \; \forall i\in N,j\in M, q \in Q &&\label{c1}\\
G_{i}(j) p_{ijq} - \sum_{h: \, b(h)\not= b(i), \, v \in Q} SIR(q) G_{i}^{b(h)}(j) p_{hjv}\geq SIR(q)BN_0 (1-L(1-x_{ijq}))&&\label{sir}\\
\forall i\in N,j\in M, q \in Q&&\nonumber\\
\sum_{i \in N_k, q\in Q}x_{ijq}\leq 1 \;\;\;\;\forall j\in M, k\in C &&\label{c2}\\
\sum_{j\in M, q\in Q}B\eta_0 q x_{ijq}\geq B\eta_0 r_i\;\;\;\;\forall i \in N &&\label{c3}\\
p_{ijq} \geq {SIR(q)BN_0 \over G_{i}(j)} x_{ijq}\;\;\;\; \forall i\in N,j\in M, q \in Q  &&\label{ridondante1}\\
p_{ijq}\geq 0\;\;\;\;\forall i\in N,j\in M, q \in Q &&\\
x_{ijq}\in \{0,1\}\;\;\;\;\forall i\in N,j\in M, q \in Q&&
\label{clast}
\end{eqnarray}
} The objective function accounts for the overall transmission
power. In Constraints (\ref{c1}) and (\ref{sir}), $L$ is a
suitable large positive number. Constraints (\ref{c1}) are logic
constraints, forcing power $p_{ijq}$ to be $0$ if subcarrier $j$
is not assigned to user $i$ with format $q$. In according with
Equations (\ref{sist1r}), Constraints (\ref{sir}) state that if
user $i$ is assigned subcarrier $j$ with format $q$ (i.e., if
$x_{ijq}=1$) power $p_{ijq}$ cannot be smaller than
${SIR(q)(\sum_{h \in N: \, b(h)\not= b(i), v \in Q}
G_{i}^{b(h)}(j) p_{hjv} + BN_0) \over G_{i}(j)}$. On the other
hand, if $x_{ijq}=0$, the right term of (\ref{sir}) is a large
negative number, and Constraints (\ref{sir}) are always satisfied.
Constraints (\ref{c2}) state that at most one user per cell (using
only one transmission format) can be assigned to a given
subcarrier. Constraints (\ref{c3}) require that a rate of at least
$R_i=B\eta_0 r_i$ is assigned to each user $i$. Recall that, if
subcarrier $j$ is assigned to user $i$ with format $q$, a rate of
$B\eta_0q$ is assigned to user $i$. Constraints (\ref{c3}) can be
divided by the term $B\eta_0$. Finally, Constraints
(\ref{ridondante1}) are redundant, but improve the solution of the
linear relaxation used by the algorithm introduced in Section
\ref{sec:algo}. They state that if user $i$ is assigned to
subcarrier $j$ with format $q$, the transmission power $p_{ijq}$
cannot be smaller than ${SIR(q)BN_0 \over G_{i}(j)}$
(corresponding to the case in which $i$ is the unique user
assigned to subcarrier $j$). In Section \ref{sec:exp}, we solve
the above MILP formulation on randomly generated RRAP problems.

\section{Problem complexity}\label{sec:comp}
In this section, we show that RRAP is strongly $NP$-hard even if a
single cell  exists (i.e., $c=1$) and only two transmission
formats can be used. In particular, the following theorem holds.
\begin{theorem} $RRAP$ is strongly NP-hard even
when $c=1$ and the set of the available transmission formats is
$Q=\{1, \bar q\}$.
\end{theorem}
\begin{proof}
Consider the case in which only one single cell exists. Note that,
in this case, by Constraints \eqref{c2}, at most one user can be
assigned to each subcarrier. Problem $RRAP$ in its decision form
can be stated as follows. We are given a set $N$ of $n$ users,
each requiring a given transmission rate, a set $M$ of
subcarriers, a set of transmission formats $Q$ and a
real value $\alpha$. All the users belong to the same cell.\\
\textbf{Question:} Is there an assignment of users to subcarriers
and  of transmission formats to each user-subcarrier pair such
that $(a)$ at most one user can be assigned to each subcarrier,
$(b)$ the transmission requirements of the users are fulfilled,
and $(c)$ the total
transmission power does not exceed $\alpha$?\\
The proof is by reduction from the strongly NP-hard scheduling
problem $P | M_j |C_{max} \leq 2$ \cite{Verhaegh}. According to
the notation for machine scheduling problems \cite{brucker},
problem $P |M_j |C_{max} \leq 2$ can be described as follows.
Given a set of $n$ identical parallel machines of capacity
$C_{max}=2$,  a set of jobs with processing times 1 or 2 (in the
following called lengths), and for each job $j$ a set of machines
$M_j$ able to process it, the problem is of finding, if it is
possible, an assignment of the jobs to the machines in such a way
that ($a$) the sum of the lengths of the jobs assigned to a given
machine does not exceed the machine capacity $C_{max}$, and ($b$)
each job $j$ is assigned to exactly one  machine in $M_j$.\\
Given an instance of $P |M_j |C_{max} \leq 2$, let $F_1$ and $F_2$
be the number of jobs in the instance of length 1 and 2,
respectively. Without loss of generality, we can suppose that
$F_1+2F_2=2n$, i.e., we can restrict to consider instances in
which all machine capacity is used in any feasible assignment. In
fact, a feasible instance in which $F_1+2F_2<2n$ can be
transformed into a feasible instance where $F_1+2F_2=2n$, by
adding $2N-F_1-2F_2$ jobs of length 1 that
can be processed by all the $n$ machines (hence, for such jobs we have that $M_j$ is equal to the set of machines).\\
 Given an instance of $P |M_j |C_{max}
\leq 2$, we build an instance $I$ of $RRAP$ as follows. In $I$,
two transmission formats exists, namely, $Q=\{1,\bar q\}$, where
$\bar q>1$ is a suitable value. In $I$, the number of  subcarriers
is equal to the number of jobs. Subcarriers are of two types. In
particular, for each job $j$ of length 1 or 2, we introduce a
subcarrier $j \in M$ of type 1 or 2, respectively. Users
correspond to machines. Hence, $n$ users are considered in $I$.
Each user requires at least a rate $R_i=2B\eta_0$ (i.e. $r_i=2$),
for $i=1,\ldots,n$. For each subcarrier $j$, let $A_j$ be the set
containing the users that
corresponds to the machines in $M_j$ (containing the machines able to process job $j$).\\
 Given a subcarrier $j$ of type 1, for each user $i$,  $G_i(j)$ is set in such a way that
$P_{ij1}= \frac{\alpha F_1}{ F_1+F_2}$ if $i \in A_j$ and
$P_{ij1}>\alpha$ if $i \notin A_j$. Format $\bar q$ is chosen in
such a way that $P_{ij\bar q}>\alpha$ for all users $i$ and
subcarriers $j$ of type 1. Hence, the subcarriers of type 1 can be
used only with format 1 (providing a total transmission rate of
$F_1$), in any solution with total power not greater than
$\alpha$.
 For each subcarrier $j$ of type 2, $G_{i}(j)$ is chosen in such a way that $P_{ij\bar q}=
\frac{\alpha F_2}{ F_1+F_2}$ ($P_{ij1}$ is obviously smaller) if
$i \in A_j$ and $P_{ij\bar q}>\alpha$ if $i \notin A_j$.\\
Note that, to get a solution with total power at most $\alpha$,
the $F_1$ subcarriers of type-1 can be only used with format 1. As
a consequence and since $F_1+2F_2=2n$, the subcarriers of type 2
must be necessarily used with format $\bar q$ to satisfy the
transmission requirements of all the users (i.e., constraints
\eqref{c3} with $r_i=2$). Hence, an assignment of subcarriers to
users of total power $\alpha$ in $I$, if any, assigns each
subcarrier $j$ of type 1 (of type 2) to a user of $A_j$ with
format 1 (format $\bar q$). Since $r_i=2$, either two subcarriers
of type 1 or one subcarrier of type 2 (getting a rate $B\eta_0
\bar q \geq 2B\eta_0$) are assigned to each user $i$. Such an
assignment corresponds to a feasible solution for the instance of
problem $P | M_j |C_{max} \leq 2$. On the other hand, if no
feasible solution exists for the $P | M_j |C_{max} \leq 2$
instance, by construction, no assignment of subcarriers to users
exists in $I$ with total power equal to $\alpha$ or smaller.
Hence, a feasible solution exists for the $P | M_j |C_{max} \leq
2$ instance if and only if an assignment of subcarriers to users
exists in $I$ with total power equal to $\alpha$, and the thesis
follows.\qed\end{proof}

\section{Heuristic algorithms for RRAP}\label{sec:algo}

Usually, small computational times (about few tens of
milliseconds) are required for solving RRAP, so that an exact
approaches can not used in practice for solving the problem. In
this section, two heuristic algorithms are proposed. The two
heuristics, called H-LP and H-LAGR, are an extension and an
improvement of two algorithms from the
literature\cite{MultiAssign,Moretti}, designed for solving RRAP
when only one single transmission format is considered (i.e.
$|Q|=1$).

\subsection{Algorithm H-LAGR}
H-LAGR is an iterative algorithm based on a network flow approach.
It consists in iteratively solving a relaxation of the MILP
formulation (\ref{fob})--(\ref{clast}), as described in the
following. When only one single transmission format is considered,
it is easy to note that (see for example \cite{MultiAssign}), the
MILP formulation obtained by relaxing the {\em interference
constraints} (\ref{sir}) in the Lagrangian way (see, for example,
\cite{netflow} for the description of the Lagrangian relaxation
technique) of the formulation of Section \ref{sec:ILPform} can be
solved in polynomial time as a minimum cost network flow problem.
In fact, when constraints (\ref{sir}) are relaxed and only one
single transmission format is used, the formulation can be
decomposed into $c$, i.e., one per cell, minimum cost flows
problems, where the problem formulation related to cell $k\in C$
is reported in the following. Note that, since a single
transmission format exists, we use variables $x_{ij}$ and
$p_{ij}$, where $x_{ij}$ is 1 if user $i$ is assigned to
subcarrier $j$ (with the unique transmission format) and 0
otherwise and variable $p_{ij}$ is the related transmission power.
\begin{eqnarray}\label{cell_k}
 \min  \sum_{i\in N_k,j\in M}p_{ij}&& \label{cc1}\\
p_{ij}\leq L \, x_{ij}&&\forall i \in N_k, j\in M \label{cc3}\\
\sum_{i \in U_k}x_{ij}\leq 1&&\forall j \in M \label{cc4}\\
\sum_{j}x_{ij}= r_i&&\forall i\in N_k \label{cc5}\\
p_{ij} \geq {SIR BN_0 \over G_{i}(j)} x_{ij}&&i \in N_k, j\in M \label{cc7}\\
p_{ij}\geq 0&&\forall i \in N_k, j\in M \label{cc8}\\
x_{ij}\in \{0,1\}&&\forall i \in N_k, j\in M \label{cc9}
\end{eqnarray}

Observe that, by constraints \eqref{cc7}, variables $p_{ij}$ can
be replaced by ${SIR BN_0 \over G_{i}(j)} x_{ij}$ throughout the
formulation, and it can be rewritten by only using variables
$x_{ij}$. Hence, the term ${SIR BN_0 \over G_{i}(j)}$ can be
viewed as the "cost" of assigning user $i$ to  subcarrier $j$. It
is easy to see that formulation \eqref{cc1}--\eqref{cc9} can be
solved as a minimum cost network flow problem \cite{MultiAssign}.

Each iteration of H-LAGR is composed of two phases. In the first
phase, the $c$ minimum cost network flow problems are
simultaneously solved to get a user-to-subcarrier assignment
(using a single transmission format). In the second phase,  the
solution found in the first phase (that may be not feasible for
the original problem, since the {\em interference constraints}
(\ref{sir}) are ignored) is "adjusted". In particular, the
following two types of adjustments are considered in the second
phase.

- Users' removal: If under the current assignment, users
transmitting on a given subcarrier $j$ requires high transmission
power levels, or if system (\ref{sist1r})--(\ref{sist12r}) for
subcarrier $j$ is not feasible, a peeling procedure is used to
remove some users from that subcarrier (the removed users are
those that make smaller the transmission power of the subcarrier).
At the same time, the "cost" of the removed user, say $i$, (for
subcarrier $j$) is increased to

$\lambda_{ij}{SIR BN_0 \over G_{i}(j)}$

\noindent where $\lambda_{ij}$ is a suitable value greater than 1.
Such cost updating is performed to make less profitable the
assignment of the user $i$ to the subcarrier $j$, at the next
iteration of H-LAGR.

- Transmission format adjustment: We illustrate this procedure
through an example. Suppose that, during the users' removal
procedure, a user $i$ is removed from subcarrier $j$, while it
still transmits on some subcarriers, say  for example $j'$ and
$j''$. (Hence, after the removal procedure, user $i$ does not
satisfy her transmission requirements.) In this second procedure,
the algorithm increases, if it possible, the transmission format
of user $i$ on the subcarriers $j'$ and $,j''$, until the
transmission rate requirements of the user are fulfilled.

H-LAGR stops when a prefixed number of iterations is reached
providing the best solution found so far (i.e. that maximizing the
overall transmission rate and minimizing  the total transmission
power).

\subsection{Algorithm H-LP}
H-LP is based on a decentralized approach \cite{Moretti}.
Also in H-LP, an iterative scheme is followed in which the single
format allocation problem \eqref{cc1}--\eqref{cc9} is separately
solved on each cell. At each iteration, all cells in the system
change their allocations simultaneously. The power costs $p_{ij}$
are updated on the base of the interference measured at the end of
the previous iteration.
The algorithm stops when a steady-state is reached (i.e., in which
no cell has interest in changing its allocation). As described in
\cite{Moretti}, the convergence of this decentralized algorithm is
not guaranteed and thus after a certain number of iterations the
rate requirements of the users are progressively reduced
determining a certain rate loss with respect to the
initial targets.

\section{Experimental results}\label{sec:exp}

In this section, preliminary experimental results on random
generated instances are presented. Instances have been generated
as in \cite{MultiAssign}. In all the instances the number of cells
is $c=7$, containing the same number of users, the number of
subcarriers is $m=16$, the overall signal bandwidth is $B_{tot}=5$
MHz and the channel is frequency selective Rayleigh fading with an
exponential power delay profile. The rms delay spread is
$\sigma_{\tau} = 0.5~\mu s$, typical of a urban environment.  We
also assume a fixed throughput per cell evenly shared among the
$|N_k|=n_k$ users, which are uniformly distributed in hexagonal
cells of radius $R=500$ meters. Hence, for each user $i \in C_k$,
$r_i=m/n_k$ is set.  Three classes of instances have been
generated varying $n_k$. In particular, we set $n_k=2,4,8$, so
that $r_i=8,4,2$, respectively, for each user in $N$, in the
different classes. For each value of $n_k$, 10 instances have been
generated through simulation of realistic scenarios.

The performances of the heuristics H-LAGR and H-LP have been
compared with a truncated branch and cut algorithm that uses the
Integer Linear Programming formulation of Section
\ref{sec:ILPform} solved with \textsc{Cplex} 9.1. All the
experiments have been performed on a 1.6GHz Pentium M laptop
equipped with 1GB RAM.

In Table \ref{TabRes}, a performance comparison of H-LP, H-LAGR
and CPLEX on the MILP formulation of Section \ref{sec:ILPform} is
given. In CPLEX, a limit of 100,000 branch and bound nodes has
been set. In the table, the first column reports on the value of
$n_k$. For each value of $n_k$, the results  are an average on the
10 instances. For each algorithm, $pow.$ is the total transmission
power (in $Watt$), \%$rate$ $loss$ is the percentage of not
assigned required transmission rate (respect to the total number
of required sub-carriers in the 10 instances), and $time$ is the
average computational time in seconds.

In all the instances, the three algorithms find solutions
satisfying all the requirements on the transmission rate. The
transmission powers in the solutions found  of H-LP and H-LAGR are
very similar (H-LP finds slightly better solutions) and are bigger
than those found by CPLEX. However, the solutions values found by
the two heuristics are quite close to the values found by CPLEX,
especially in the instances with 4 and 8 users per cell.

From the computational time point of view, we have that H-LP and
H-LAGR are very fast (about tens of milliseconds on average) while
CPLEX requires more than 2500 seconds, on average. Observe that,
H-LAGR requires higher computational times on instances with
smaller values of $n_k$. This is due to the "Transmission format
adjustment" procedure that requires more calls when $n_k$ is
smaller.

\begin{table}[htb]
{\center
\begin{tabular}{|c||c|c|c||c|c|c||c|c|c|}
\hline &\multicolumn{3}{|c||}{H-LP}& \multicolumn{3}{|c||}{H-LAGR}&\multicolumn{3}{|c|}{CPLEX} \\
\hline $n_k$   &$pow$.&\%$rate$&$time$&$pow$.&\%$rate$&$time$&$pow$.&\%$rate$&$time$\\
 & &$loss$& & & $loss$& && $loss$ &\\
\hline2 &  64.510 & 0& 0.020&60.901&  0 &0.089 &40.044& 0&514.8\\
\hline4  & 37.128 & 0& 0.022& 38.535 & 0 &  0.020&32.128&0&1562\\
\hline8 &26.518&0& 0.024&27.809&  0& 0.012&25.233&0&5631.7\\
\hline
\end{tabular}
\caption{Comparison results on random instances.}\label{TabRes}
}
\end{table}

\section{Conclusions}\label{sec:conc}
In this paper, we addressed a radio resource allocation problem
arising in wireless cellular networks. We  study the computational
complexity of the  problem and proposed an efficient heuristics
algorithms for its solution. A preliminary computational study
shows that the algorithms are suitable for real world applications
both for the solution quality and for the short computational
times.

\end{document}